\documentclass{amsart}
\usepackage{amssymb,times,array,xcolor}

\newcounter{lemma}
\newtheorem{Theorem}{Theorem}
\newtheorem{Lemma}[lemma]{Lemma}
\newtheorem{Corollary}[lemma]{Corollary}
\newtheorem{Proposition}[lemma]{Proposition}
\newtheorem{theorem}{Theorem}

\theoremstyle{definition}

\def\F{\mathbb F}

\def\Q{\mathbb Q}

\def\Z{\mathbb Z}

\def\fp{\mathfrak p}

\def\GL{\mathrm{GL}}
\def\JS#1#2{\left(\frac{#1}{#2}\right)}
\def\Nm{\mathit{Nm}}

\def\Gal{\mathrm{Gal}}

\def\mod{\ \mathrm{mod}\ }

\begin{document}

\title{Jacobsthal identity for $\Q(\sqrt{-2})$}
\author{Ki-Ichiro Hashimoto}
\address{Department of Mathematics, Waseda University, 3-4-1, Okubo
  Shinjuku-ku, Tokyo 169, Japan}
\email{khasimot@waseda.ac.jp}

\author{Ling Long}
\address{Department of Mathematics, Iowa State University, Ames, IA
  50011, USA}
\email{linglong@iastaate.edu}

\author{Yifan Yang}
\address{Department of Applied Mathematics, National Chiao Tung
  University and National Center for Theoretical Sciences, Hsinchu,
  Taiwan}
\email{yfyang@math.nctu.edu.tw}

\begin{abstract} Let $p$ be a prime congruent to $1$ or $3$ modulo $8$
  so that the equation $p=a^2+2b^2$ is solvable in integers. In this
  paper, we obtain closed-form expressions for $a$ and $b$ in
  terms of Jacobsthal sums. This is analogous to a classical
  identity of Jacobsthal.
\end{abstract}

\date\today

\thanks{Long was supported  by  NSA grant H98230-08-1-0076. Long and
  Yang would like to thank National Center for Theoretical Sciences in
  Hsinchu Taiwan at which place the project was initiated. Yang is
  supported by NSC Grant 97-2115-M-009-001. Part of this work was done
  while he visited the first author at the Waseda University. He would
  like to thank the first author and the Waseda University for the
  enormous hospitality.}

\maketitle

\begin{section}{Introduction}

Let $p$ be an odd prime. Recall that a famous result of Fermat states
that $p$ is the sum of two integer squares,  i.e., $p=a^2+b^2$ for
some integers $a$ and $b$, if and only if $p\equiv 1\mod 4$. It turns
out there is a closed-form formula for the integers $a$ and $b$ in
terms of the Legendre symbol $\JS\cdot p$. This is the classical
identity of Jacobsthal.

\begin{theorem}[Jacobsthal] \label{theorem: Jacobsthal}
  Let $p$ be a prime congruent to $1$ modulo
  $4$, and $n$ be a quadratic nonresidue modulo $p$. Set
  \begin{equation} \label{equation: Jacobsthal}
    A=\frac12\sum_{x=0}^{p-1}\JS{x^3-x}p, \qquad
    B=\frac12\sum_{x=0}^{p-1}\JS{x^3-n x}p.
  \end{equation}
  Then $A,B\in\Z$ and $A^2+B^2=p$. More precisely, if, for a prime $p$
  congruent to $1$ modulo $4$, we let $a$ be an odd integer and $b$ be
  an even integer such that $p=a^2+b^2$, then
  $$
    \sum_{x=0}^{p-1}\JS{x^3-n x}p=\begin{cases}
    \pm 2a, &\displaystyle\text{if }p\equiv 1\mod 4\text{ and }\JS{n}p=1, \\
    \pm 2b, &\displaystyle\text{if }p\equiv 1\mod 4\text{ and }
    \JS{n}p=-1, \\
    0, &\text{if }p\equiv 3\mod 4. \end{cases}
  $$
\end{theorem}

In general, we shall refer to a sum of the form
$$
  \sum_{x=0}^{p-1}\JS{f(x)}p, \qquad f(x)\in\Z[x]
$$
as a \emph{Jacobsthal sum}.

There are many proofs for Jacobsthal's identity. Gauss supplied an
elementary proof using only basic properties of the Legendre symbols.
(See \cite[Page 91]{Modular123}.) One can also utilize properties of
Jacobi sums to prove Theorem \ref{theorem: Jacobsthal}. (See
\cite[Page 190]{Berndt-Evans-Williams}.) On the other hand, the sums
in \eqref{equation: Jacobsthal} have the obvious meaning of counting
points on the elliptic curves $y^2=x^3-x$ and $y^2=x^3-n x$ over
$\F_p$, respectively, so it is possible to use tools from arithmetic
geometry to give yet another proof. In Section \ref{section: AG to
  Jacobsthal}, we will briefly explain this approach.

Now observe that the equation $p=a^2+b^2=(a+bi)(a-bi)$ can be regarded
as the prime factorization of $p$ in the ring of integers $\Z[i]$ in
the number field $\Q(i)$. Naturally, one may ask whether analogous
identities exist in the case of other number fields. Let
$\Q(\sqrt{-D})$ be an imaginary quadratic number field with
discriminant $-D$. If $\Q(\sqrt{-D})$ has class number one, then
whether a prime $p$ splits in $\Q(\sqrt{-D})$ depends solely on the
value of $\JS{-D}p$. That is, if $\JS{-D}p=1$, then there are integers
$a$ and $b$ such that $p=f_D(a,b)$, where
$$
  f_D(x,y)=\begin{cases}
  \displaystyle x^2+xy+\frac{1+D}4y^2, &\text{if }D\text{ is odd}, \\
  \displaystyle x^2+\frac{D}4y^2, &\text{if }D\text{ is even}
  \end{cases}
$$
is the principal form of discriminant $-D$. Then one can ask whether
these integers $a$ and $b$ can be expressed as Jacobsthal sums in a
uniform way. For the case $D=3$, it is relatively easy. Using Jacobi
sums, we find the following analogue of Theorem \ref{theorem:
  Jacobsthal}.

\begin{theorem}[{Chan-Long-Yang \cite{Cubic}}] \label{theorem: CubicJ}
Let $p$ be a prime satisfying $p\equiv 1\mod 6.$ Suppose $n$ is
any integer such that $x^3\equiv n \mod p$ is
not solvable. Then
$$
  3p=A^2+AB+B^2,
$$
where
$$
  A=\sum_{x=0}^{p-1}\left(\frac{x^3+1}{p}\right) \quad\text{and}\quad
  B=\left(\frac n{p}\right)\sum_{x=0}^{p-1}
    \left(\frac{x^3+n}{p}\right).
$$
\end{theorem}

The main purpose of this paper is to prove an analogue of Jacobsthal's
identity in the case $D=8$. Note that $\JS{-2}p=1$ if and only if
$p\equiv 1,3\mod 8$. For such a prime $p$, there exist integers $a$
and $b$ such that $p=a^2+2b^2$.

\begin{Theorem} \label{theorem: main}
  Let $p$ be a prime congruent to $1$ or $3$ modulo $8$. Let
  \begin{equation} \label{equation: A}
    A=\frac12\sum_{x=0}^{p-1}\JS{x^3+4x^2+2x}p.
  \end{equation}
  Moreover,
  \begin{enumerate}
  \item when $p\equiv 1\mod 8$ and $n$ is a quadratic nonresidue modulo
    $p$, set
  \begin{equation} \label{equation: B 1}
    B=\frac14\sum_{x=0}^{p-1}\JS{x^5+nx}p,
  \end{equation}
  and
  \item when $p\equiv 3\mod 8$, set
  \begin{equation} \label{equation: B 2}
  B=\frac14\left(1+\sum_{x=0}^{p-1}\JS{x^6+4x^5+10x^4-20x^2-16x-8}p
    \right).
  \end{equation}
  \end{enumerate}
  Then $A$ and $B$ are integers and satisfy $A^2+2B^2=p$.
\end{Theorem}

One remarkable feature of the main theorem is the existence of
polynomials $f(x)$ and $g(x)$ in $\Z[x]$ such that the integers $A$
and $B$ in $p=A^2+2B^2$ can be expressed as Jacobsthal sums associated
to $f(x)$ and $g(x)$, respectively, for \emph{all} primes $p$
congruent to $3$ modulo $8$.

Our approach is mainly arithmetic-geometric. The elliptic curve
$y^2=x^3+4x^2+2x$ corresponding to the Jacobsthal sum in
\eqref{equation: A} has complex multiplication by the order
$\Z[\sqrt{-2}]$. Thus, by a famous theorem of Deuring (see
\cite[Theorem II.10.5]{Silverman-adv}), its
$L$-function is the same as the $L$-function of a Hecke
Gr\"ossencharakter on $\Q(\sqrt{-2})$. It is straightforward to verify
that the quantity $A$ in \eqref{equation: A} has the same absolute
value as the integer $a$ in $p=a^2+2b^2$. The two hyperelliptic curves
$y^2=x^5+nx$ and $y^2=x^6+4x^5+10x^4-20x^2-16x-8$ corresponding to
the Jacobsthal sums in \eqref{equation: B 1} and \eqref{equation: B 2}
are both isomorphic to $y^2=x^5+x$, although the two isomorphisms are
over two different number fields. Thus, one can deduce information
about the $L$-functions of the two hyperelliptic curves from that of
$y^2=x^5+x$. The details will be carried out in Section \ref{section:
  proof}.
\end{section}

\begin{section}{Arithmetic-geometric approach to Jacobsthal's
    identity} \label{section: AG to Jacobsthal}

In this section, we will present an arithmetic-geometric proof of
Theorem \ref{theorem: Jacobsthal}. This will serve as an illustrating
example how one can obtain information about the $L$-function of an
algebraic curve over $\Q$ from that of another algebraic
curve over $\Q$, assuming that the two curves are isomorphic over a
number field.

For a nonzero integer $n$, let $E_n$ denote the elliptic curve
$y^2=x^3-n x$. Clearly, we have, for a prime $p$ relatively
prime to $2n$,
$$
  \#E_n(\F_p)=p+1+\sum_{x=0}^{p-1}\JS{x^3-nx}p.
$$
Therefore, the reciprocal of the $p$-factor of $L(E_n/\Q,s)$ is
equal to
$$
  1+\left(\sum_{x=0}^{p-1}\JS{x^3-n x}p\right)p^{-s}+p^{1-2s}.
$$
For the case $n={1}$, the $L$-function $L(E_{1}/\Q,s)$ is well-known.

\begin{Lemma} \label{lemma: L for E} We have
$$
  L(E_{1}/\Q,s)=\prod_{p\equiv 1\mod 4}\frac1{1-2\epsilon_pa_p p^{-s}+p^{1-2s}}
  \prod_{p\equiv 3\mod 4}\frac1{1+p^{1-2s}},
$$
where for $p\equiv 1\mod 4$, $a_p$ and $b_p$ are positive integers with
$a_p$ odd and $b_p$ even such that $p=a_p^2+b_p^2$, and
$$
  \epsilon_p=\JS{-1}{a_p}(-1)^{b_p/2}.
$$
\end{Lemma}

\begin{proof} See \cite[Page 59]{Koblitz}. (Note that in
  \cite{Koblitz}, the $L$-function $L(E_{1}/\Q,s)$ is described
  differently, but it is easy to check that it gives the same
  $L$-function as above.)
\end{proof}

We now extract informations about $L(E_n/\Q,s)$ from the above
lemma using the fact that $E_{1}$ and $E_n$ are isomorphic over
$\Q(\sqrt[4]{n})$.

Recall that for a given elliptic curve $E$ defined over a number field
$K$ and an arbitrary rational prime $\ell$, one can associate to $E$ a
continuous representation  $$\rho_{E,\ell}: \Gal(\overline{K}/K)\to
\GL(2,\Q_{\ell})$$ via the Tate-module $T_{\ell}E$ of $E$. For
detailed discussions, see \cite{Silverman}. In general, when $C$ is a smooth
irreducible curve defined over $K$ of genus $g$, one can
associate to $C$ a $2g$-dimensional representation of
$\Gal(\overline{K}/K)$ by considering the Tate module of the
Jacobian of $C$. The $L$-function $L(\rho_{C,\ell},s)$ of
$\rho_{C,\ell}$ is defined by  local Euler factors. Let  $\mathcal
O_K$ be the ring of integers of $K$. For any prime ideal $\fp$ of
$\mathcal O_K$ at which $\rho_{C,\ell}$ is unramified,    the
arithmetic Frobenius $\text{Frob}_\fp$ acts on the representation space
of $\rho_{C,\ell}$ with characteristic polynomial $P_{\fp}(T)$ of
degree $2g$.  Then the local Euler $\fp$-factor of
$L(\rho_{C,\ell},s)$ is $P_{\fp}(q^{-s})^{-1}$ where $q=|\mathcal
O_K/\fp|$.
Moreover, the Hasse-Weil zeta function of the reduction of $C$ modulo
$\fp$ is equal to
$$
  \frac{P_{\fp}(t)}{(1-t)(1-qt)}.
$$
(cf. \cite{Milne}.)
From now on, by the $L$-function $L(C,s)$  we mean the $L$-function
$L(\rho_{C,\ell},s)$ for any rational prime $\ell$.

Now let us first recall a property of group representations.

\begin{Lemma} \label{lemma: representation in cyclic}
  Let $G$ be a group and $H$ be a normal subgroup of $G$ of
  finite index such that $G/H$ is cyclic. Assume that
  $\rho_1:G\to\GL(V_1)$ and $\rho_2:G\to\GL(V_2)$ are two irreducible
  representations over an algebraically closed field of characteristic
  not dividing $|G/H|$ such that the restrictions of $\rho_1$ and
  $\rho_2$ to $H$ are isomorphic. Then $\rho_1=\rho_2\otimes\chi$ for
  some representation $\chi$ of $G$ of degree $1$ that is lifted from a
  character of $G/H$.
\end{Lemma}

\begin{proof} See \cite{Clifford}.
\end{proof}

We now give a proof of Theorem \ref{theorem: Jacobsthal}.

\begin{proof}[Proof of Theorem \ref{theorem: Jacobsthal}]
  Let $n$ be a nonsquare integer. Let $E_{1}$ and $E_n$ denote the
  elliptic curves $y^2=x^3-x$ and $y^2=x^3-n x$. They are isomorphic
  over $\Q(\sqrt[4]{n})$, which is not Galois over $\Q$. Note that
  both curves have complex multiplication by the ring of Gaussian
  integers $\Z[i]$ as sending $(x,y)$ to $(-x,iy)$ is an automorphism
  on both curves. Consequently,
  $$
    \sum_{x=0}^{p-1}\JS{x^3-n x}p=0
  $$
  when $p\equiv 3\mod 4$. For any rational prime $\ell$, the Galois
  representations $\rho_{E_{1},\ell}$ and $\rho_{E_n,\ell}$ of
  $\text{Gal}(\overline{\Q}/\Q)$ are absolutely irreducible while
  their restrictions  to $\text{Gal}(\overline{\Q}/\Q(i))$ decompose
  as
  $$
    \rho_{E_{1},\ell}\big|_{\text{Gal}(\overline{\Q}/\Q(i))}=\pi_{1}\oplus
    \bar \pi_{1}, \qquad
    \rho_{E_n,\ell}\big|_{\text{Gal}(\overline{\Q}/\Q(i))}=\pi_n\oplus \bar
    \pi_n,
  $$
  where $\pi_{1},\pi_n$ are 1-dimensional representations of
  ${\text{Gal}(\overline{\Q}/\Q(i))}$ and $\bar \pi_{1},\bar \pi_n$
  their  complex conjugates respectively. Since  $E_{1}$ and $E_n$ are
  isomorphic over $L=\Q(i,\sqrt[4]{n})$,
  $$
    \rho_{E_{1},\ell}\big|_{\Gal(\overline{\Q}/L)}
   =\rho_{E_n,\ell}\big|_{\Gal(\overline{\Q}/L)}.
  $$
  Without loss of generality we may assume that
  $\pi_{1}|_{\Gal(\overline{\Q}/L)}=\pi_n|_{\Gal(\overline{\Q}/L)}.$
  Then by Lemma \ref{lemma: representation in cyclic},
  $\pi_1=\pi_n\otimes\chi$ where $\chi$ is a character
  $\Gal(\overline\Q/\Q(i))$ lifted from a character of 
  $\Gal(L/\Q(i))\cong \Z/4\Z$ with 
  kernel $\Gal(\overline\Q/L)$ as $L=\Q(i,\sqrt[4]{n})$ is the smallest
  field over which the two curves are isomorphic. When $p\equiv 1\mod
  4$ and $p\nmid 2n\ell$, $\F_p(\sqrt[4]{n})$ is a quartic extension
  of $\F_p$ if and only if $\left ( \frac{n}p\right )=-1$, i.e.
  $\chi(\text{Frob}_p)=\pm i$ if and only if $\left ( \frac{n}p\right
  )=-1$. By Lemma \ref{lemma: L for E},
  $\text{Re}\{\pi_1(\text{Frob}_p)\}=\pm a$. Therefore,
  $\text{Re}\{\pi_n(\text{Frob}_p)\}=\text{Re}\{\pi_1(\text{Frob}_p)\otimes
  \chi(\text{Frob}_p)\}=\pm a$ if $\left ( \frac{n}p\right )=1$ and
  $\text{Re}\{\pi_n(\text{Frob}_p)\}=\pm b$ otherwise. This proves the
  theorem.
\end{proof}
\end{section}

\begin{section}{Proof of Theorem \ref{theorem: main}} \label{section: proof}

In this section, we will prove Theorem \ref{theorem: main}. As
explained in the introduction section, since the elliptic curve
$y^2=x^3+4x^2+2x$ has complex multiplication by $\Z[\sqrt{-2}]$, its
$L$-function is easy to write down in terms of a Hecke
Gr\"ossencharakter on $\Q(\sqrt{-2})$. This is done in Section
\ref{subsection: elliptic curve}. We then show that the hyperelliptic
curve $y^2=x^6+4x^5+10x^4-20x^2-16x-8$ is isomorphic to $y^2=x^5+x$
over a Kummer extension $L$ of $\Q(e^{2\pi i/8})$ and study the
$L$-function of $y^2=x^5+x$ in Section \ref{subsection:
  hyperelliptic}. In Section \ref{subsection: number field}, we will
obtain the Hecke Gr\"ossencharakters associated to the cyclic
extension $L/\Q(e^{2\pi i/8})$.
Finally, we will give a proof of Theorem \ref{theorem: main} in the
last section.

\begin{subsection}{The elliptic curve $y^2=x^3+4x^2+2x$}
\label{subsection: elliptic curve}

\begin{Lemma} \label{lemma: L for E8} Let $E:y^2=x^3+4x^2+2x$. The
  $L$-function $L(E/\Q,s)$ is given by
$$
  \prod_{p\equiv1,3\mod 8}\frac1{1-2\epsilon_p a_pp^{-s}+p^{1-2s}}
  \prod_{p\equiv 5,7\mod 8}\frac1{1+p^{1-2s}},
$$
where $a_p$ and $b_p$ are positive integers such that $p=a_p^2+2b_p^2$
and
$$
  \epsilon_p=\begin{cases}
  2(-1)^{b_p/2}\JS{-2}{a_p}, &\text{if }p\equiv 1\mod 8, \\
 -2\JS{-2}{a_p}, &\text{if }p\equiv 3\mod 8.
  \end{cases}
$$
\end{Lemma}

\begin{proof} The elliptic curve $E:y^2=x^3+4x^2+2x$ has complex
  multiplication by $\Z[\sqrt{-2}]$ and its conductor is $256$. (See
  \cite[Page 483]{Silverman-adv}.) Thus, by a well-known result of
  Deuring, the $L$-function $L(E/\Q,s)$ is identical with the
  $L$-function $L(\chi,s)$ of a Hecke Gr\"ossencharakter $\chi$ of the
  field $\Q(\sqrt{-2})$ of conductor $(\sqrt{-2})^5$. (See Theorem
  II.10.5 and Corollary II.10.5.1 of \cite{Silverman-adv}.) It is not
  difficult to work out this Hecke character. Namely, for each
  $a+b\sqrt{-2}\in\Z[\sqrt{-2}]$, there are unique integers $k,m,n$
  with $0\le k,m<2$ and $0\le n<4$ such that
  $$
    a+b\sqrt{-2}\equiv (-1)^k3^m(1+\sqrt{-2})^n\mod(\sqrt{-2})^5.
  $$
  Then the Hecke character $\chi$ is defined by
  $$
    \chi(a+b\sqrt{-2})=(-1)^{k+n}(a+b\sqrt{-2}).
  $$
  (Note that since $\Q(\sqrt{-2})$ has class number one, we can define
  a Hecke character elementwise.) This character can be more
  succinctly written as
  $$
    \chi(a+b\sqrt{-2})=\JS{-2}a(a+b\sqrt{-2})\cdot
    \begin{cases} (-1)^{b/2}, &\text{if }b\text{ is even}, \\
    -1, &\text{if }b\text{ is odd},\end{cases}
  $$
  which yields the expression of $L(E/\Q,s)$ given in the statement of
  the lemma.
\end{proof}

\begin{Corollary} \label{corollary: a-part} We have
  $$
    \sum_{x=0}^{p-1}\JS{x^3+4x^2+2x}p=\begin{cases}
    \pm 2a, &\text{if }p\equiv 1,3\mod 8, \\
    0, &\text{if }p\equiv 5,7\mod 8,
    \end{cases}
  $$
  where $a$ and $b$ are integers such that $p=a^2+2b^2$.
\end{Corollary}

\end{subsection}

\begin{subsection}{The hyperelliptic curve $y^2=x^5+x$}
\label{subsection: hyperelliptic}

\begin{Lemma} The two hyperelliptic curves $y^2=x^5+x$ and
  $y^2=x^6+4x^5+10x^4-20x^2-16x-8$ are isomorphic over the number
  field $\Q(\theta)$, where $\theta=2^{3/4}(\sqrt 2-1)^{3/4}$.
\end{Lemma}

\begin{proof} Notice that the polynomial
$x^6+4x^5+10x^4-20x^2-16x-8$ factorizes as
$$
  (x^2-2)(x^4+4x^3+12x^2+8x+4).
$$
We make a linear transformation sending the root $\sqrt 2$ to
$\infty$ and the root $-\sqrt 2$ to $0$, i.e., setting
$$
  x=\frac{\sqrt{2}(x_1+1)}{x_1-1}, \qquad
  y=\frac{y_1}{(x_1-1)^3},
$$
we get
$$
  y_1^2=128(2+\sqrt 2)x_1(x_1^4+3-2\sqrt 2).
$$
Then let $x_1=\sqrt{\sqrt 2-1}x_2$, $y_1=y_2$, and obtain
$$
  y_2^2=128\sqrt 2(\sqrt2-1)^{3/2}(x_2^5+x_2).
$$
Finally, setting $x_2=x_3$ and $y_2=u^{1/2} y_3$, where
$u=128\sqrt2(\sqrt 2-1)^{3/2}$, we arrive at
$$
  y_3^2=x_3^5+x_3.
$$
This proves the lemma.
\end{proof}

\begin{Proposition} \label{proposition: zeta of X2}
  Let $X$ be the hyperelliptic curve $y^2=x^5+x$
  over $\Q$. Then we have
  $$
    L(X/\Q,s)=L(E_1/\Q,s)L(E_2/\Q,s),
  $$
  where $E_1$ and $E_2$ are the elliptic curves $y^2=x^3+4x^2+2x$ and
  $y^2=x^3-4x^2+2x$, respectively.
\end{Proposition}

\begin{proof} We have
$$
  x^5+x=x((x-1)^4+4x(x-1)^2+2x^2).
$$
Thus, letting
$$
  X=\frac{(x-1)^2}x, \qquad Y=\frac{y(x-1)}{x^2},
$$
we find that $X$ and $Y$ satisfy $Y^2=X^3+4X^2+2X$. In other words,
there is a $2$-to-$1$ morphism from $X$ to $E_1$ defined over $\Q$ and
hence the $\text{Gal}(\overline{\Q}/\Q)$-representation
$\rho_{E_1,\ell}$ associated to $E_1$ is a subrepresentation of
$\rho_{X,\ell}$ associated to $X$. Similarly, setting $X=(x+1)^2/x$
and $Y=y(x+1)/x^2$, we get a morphism from $X$ to $E_2$ and conclude
that $\rho_{E_1,\ell}$ is also a
$\Gal(\overline{\Q}/\Q)$-subrepresentation of $\rho_{X,\ell}$.
Since $\rho_{E_1,\ell}$ and $\rho_{E_2,\ell}$ are nonisomorphic
absolutely irreducible $\text{Gal}(\overline{\Q}/\Q)$-representations
and $\dim_{\Q_\ell}\rho_{E_1,\ell}+\dim_{\Q_\ell} \rho_{E_2,\ell}=\dim_{\Q_\ell}
\rho_{X,\ell}$, we have
$$
  \rho_{X,\ell}=\rho_{E_1,\ell}\oplus\rho_{E_2,\ell}
$$
and
$$
  L(X/\Q,s)=L(E_1/\Q,s)L(E_2/\Q,s).
$$
This proves the proposition.
\end{proof}

\begin{Corollary} \label{corollary: L for X2}
  Let the curve $X:y^2=x^5+x$ be given as above. Let
  $$
    \frac1{(1-\alpha_{p,1}p^{-s})\ldots(1-\alpha_{p,4}p^{-s})}
  $$
  be the $p$-factor of $L(X/\Q,s)$.
  \begin{enumerate}
  \item If $p\equiv 1\mod 8$, then
  $$
    \alpha_{p,j}=\JS{-2}a(-1)^{b/2}(a\pm b\sqrt{-2}),
  $$
  each with multiplicity $2$, where $a$ and $b$ are the positive
  integers such that $p=a^2+2b^2$.
  \item If $p\equiv 3\mod 8$, then $\alpha_{p,j}=\pm a\pm b\sqrt{-2}$,
    where $a$ and $b$ are integers such that $p=a^2+2b^2$.
  \item If $p\equiv 5,7\mod 8$, then $\alpha_{p,j}=\pm i\sqrt p$, each
    with multiplicity $2$.
  \end{enumerate}
\end{Corollary}

\begin{proof} This corollary follows immediately from Proposition
  \ref{proposition: zeta of X2}, Lemma \ref{lemma: L for E8}, and the
  fact that $y^2=x^3-4x^2+2x$ is a quadratic twist of
  $y^2=x^3+4x^2+2x$ by $-1$, i.e. it is isomorphic to
  $-y^2=x^3+4x^2+2x$ over $\Q(i)$.
\end{proof}
\end{subsection}

\begin{subsection}{The number field $\Q(\theta,i)$}
\label{subsection: number field}

Let $\theta=2^{3/4}(\sqrt2-1)^{3/4}$. In the previous section, we have
seen that the two hyperelliptic curves $y^2=x^5+x$ and
$y^2=x^6+4x^5+10x^4-20x^2-16x-8$ are isomorphic over $\Q(\theta)$.
As in the case of Theorem \ref{theorem: Jacobsthal}, the field
$\Q(\theta)$ is not an abelian extension of $\Q$, so in order to apply
Lemma \ref{lemma: representation in cyclic}, we change the base field
from $\Q$ to $\Q(\zeta_8)$, where $\zeta_8=e^{2\pi i/8}$ so that
$\Q(\theta,\zeta_8)$ is Galois over $\Q(\zeta_8)$. From now on,
we let
$$
  K=\Q(\zeta_8), \qquad L=\Q(\theta,\zeta_8)=\Q(\theta,i).
$$

\begin{Lemma} The field $L=\Q(\theta,i)$ is a Kummer extension of
  $K=\Q(\zeta_8)$ obtained by adjoining the fourth root of $i(\sqrt
  2-1)$ to $\Q(\zeta_8)$. The field extension $L/K$ is unramified
  outside of the place $1-\zeta_8$.
\end{Lemma}

\begin{proof} We have $(1-\zeta_8)^4=2(\sqrt 2-1)^2i$. Thus,
  $(\theta/(1-\zeta_8)^3)^4=((\sqrt 2-1)i)^{-3}$ and
  $\Q(\theta,i)=\Q(\sqrt[4]{i(\sqrt 2-1)},\zeta_8)$. Now
  $i(\sqrt 2-1)$ is a unit in $\Z[\zeta_8]$.
  Hence, the only place at which $L/K$ can possibly be ramified is the
  prime $1-\zeta_8$ lying over $2$. This proves the lemma.
\end{proof}

The main purpose of this section is to determine the 
characters of $\Gal(\overline \Q/K)$ associated to the abelian
extension $L/K$. From now on, we set
$$
  \eta=\sqrt[4]{i(\sqrt 2-1)}.
$$
(Since $L/K$ is a Kummer extension, it does not matter which fourth
root of $i(\sqrt 2-1)$ we take.) The Galois group of $L$ over $K$ is
cyclic and generated by
$$
  \sigma:\eta\mapsto i\eta.
$$
For each prime ideal $\fp$ of $\Q(\zeta_8)$ relatively prime to
$1-\zeta_8$, the Artin symbol $\left(\frac{L/K}\fp\right)$ is defined
as the unique Galois element $\sigma^j$ such that
\begin{equation} \label{equation: Artin symbol}
  \sigma^j(\eta)\equiv\eta^{\Nm(\fp)}\mod\fp,
\end{equation}
where $\Nm(\fp)$ denotes the norm of $\fp$. Then the 
characters associated to the field extension $L/K$ are defined by
\begin{equation} \label{equation: chi k}
  \chi_k(\fp)=i^{jk}, \qquad k=0,\ldots,3,
\end{equation}
where $j$ is the integer in \eqref{equation: Artin
  symbol}.

\begin{Lemma} \label{lemma: cases mod 8}
  Let $k=1,2,3$.
  \begin{enumerate}
  \item If $\fp$ is a prime of $\Q(\zeta_8)$ lying over a prime $p$
    congruent to $3$ modulo $8$, then $\chi_k(\fp)=i^k$ or $i^{-k}$.
  \item If $\fp$ is a prime of $\Q(\zeta_8)$ lying over a prime $p$
    congruent to $5$ or $7$ modulo $8$, then $\chi_k(\fp)=1$ for all $k$.
  \end{enumerate}
\end{Lemma}

\begin{proof}
  A prime $\fp$ lying over a prime $p$ congruent to $3$, $5$, or $7$
  modulo $8$ has norm $p^2$. In the case $p\equiv3\mod 8$, notice
  that
  $$
    (\sqrt 2-1)^{p+1}\equiv(-\sqrt2-1)(\sqrt 2-1)=-1 \mod p,
  $$
  which implies that
  $$
    (\sqrt2-1)^{(p^2-1)/2}\equiv-1\mod p.
  $$
  It follows that
  $$
    \eta^{\Nm(\fp)-1}=(i(\sqrt 2-1))^{(p^2-1)/4}\equiv\pm i\mod\fp,
  $$
  and $\chi_k(\fp)=\pm i^k$.

  If $p\equiv 5\mod 8$, a similar argument shows that
  $$
    (\sqrt2-1)^{(p^2-1)/4}=\left((\sqrt2-1)^{p+1}\right)^{(p-1)/4}
    \equiv(-1)^{(p-1)/4}=-1 \mod \fp
  $$
  and
  $$
    \eta^{\Nm(\fp)-1}=(i(\sqrt2-1))^{(p^2-1)/4}\equiv 1\mod \fp,
  $$
  which implies $\chi_k(\fp)=1$ for all $k=0,\ldots,3$.

  If $p\equiv 7\mod 8$, we have $\sqrt 2\equiv u\mod p$ for some
  $u\in\Z$. Then
  $$
    (\sqrt 2-1)^{(p^2-1)/4}\equiv\left((u-1)^{p-1}\right)^{(p+1)/4}
     \equiv1 \mod \fp
  $$
  and
  $$
    \eta^{\Nm(\fp)-1}=(i(\sqrt2-1))^{(p^2-1)/4}\equiv 1\mod \fp.
  $$
  Again, this gives us $\chi_k(\fp)=1$ for all $k$. This proves the
  lemma.
\end{proof}

\end{subsection}

\begin{subsection}{Proof of Theorem \ref{theorem: main}}
\label{subsection: proof}

We now prove Theorem \ref{theorem: main}. As before, we let
$$
  K=\Q(\zeta_8), \qquad L=\Q(\theta,\zeta_8),
$$
where $\theta=2^{3/4}(\sqrt 2-1)^{3/4}$. For a given number field $F$,
denote $\Gal(\overline{\Q}/F)$ by $G_F$ for convenience.

The cases $p\equiv 1\mod 8$ can be proved in a similar way as
Theorem \ref{theorem: Jacobsthal}. For instance, one can utilize
Theorem 6.2.3 of \cite{Berndt-Evans-Williams} to conclude that
$$
  \sum_{x=0}^{p-1}\JS{x^5+n x}p=\pm 4b,
$$
provided that $p$ is a prime congruent to $1$ modulo $8$, $n$ is
a quadratic nonresidue modulo $p$, and $a$ and $b$ are integers such
that $p=a^2+2b^2$. Then from Corollary \ref{corollary: a-part}, we get
the claimed identity. Alternatively, one can also follow the argument
in Section \ref{section: AG to Jacobsthal} to get the same
conclusion. We shall not give details here.

Now consider the two hyperelliptic curves $X_1:y^2=x^5+x$ and
$X_2:y^2=x^6+4x^5+10x^4-20x^2-16x-8$. Assume that the $p$-factors of
the $L$-functions of $X_1/\Q$ and $X_2/\Q$ are
$$
  \frac1{(1-\alpha_{p,1}p^{-s})\ldots(1-\alpha_{p,4}p^{-s})}, \qquad
  \frac1{(1-\beta_{p,1}p^{-s})\ldots(1-\beta_{p,4}p^{-s})},
$$
respectively. Then for $p\not\equiv 1\mod 8$, the $p$-factors of the
$L$-functions of $X_1/K$ and $X_2/K$ are
$$
  \frac1{(1-\alpha_{p,1}^2p^{-2s})^2\ldots(1-\alpha_{p,4}^2p^{-2s})^2}, \qquad
  \frac1{(1-\beta_{p,1}^2p^{-2s})^2\ldots(1-\beta_{p,4}^2p^{-2s})^2},
$$
respectively.

Counting the numbers of points on $X_2(\F_{3^n})$, we find that the
$3$-factor of $L(X_2/\Q,s)$ is
$$
  (1+4\cdot3^{-s}+8\cdot3^{-2s}+12\cdot3^{-3s}+9\cdot3^{-4s})^{-1},
$$
which implies that
$$
  \beta_{3,j}=\zeta_8(1+\sqrt{-2}),\ \zeta_8^3(1+\sqrt{-2}),\
  \zeta_8^5(1-\sqrt{-2}),\ \zeta_8^7(1-\sqrt{-2}),
$$
and
\begin{equation} \label{equation: beta}
  \beta_{3,j}^2=\pm i(1+\sqrt{-2})^2,\ \pm i(1-\sqrt{-2})^2.
\end{equation}
On the other hand, from Corollary \ref{corollary: L for X2}, we know
that
\begin{equation} \label{equation: alpha}
  \alpha_{3,j}^2=(1\pm \sqrt{-2})^2,
\end{equation}
each with multiplicity $2$.

From the above data, we proceed to determine the structure of the
semisimplification of $\rho_{X_2,\ell}|_{G_K}$ from
$\rho_{X_1,\ell}|_{G_K}$ and consider them as representations over
$\overline{\Q}_\ell$. By the discussion in Section 
\ref{subsection: hyperelliptic}, we know
$\rho_{X_1,\ell}=\rho_{E_1,\ell}\oplus\rho_{E_2,\ell}$ where
$\rho_{E_1,\ell}|_{G_K}\cong \rho_{E_2,\ell}|_{G_K}$. As $\sqrt{-2}\in
K$, $\rho_{E_1,\ell}|_{G_K}=\sigma\oplus\bar \sigma$ for some
one-dimensional representation $\sigma$ of $G_K$ and its complex
conjugate. Moreover, the restriction $\sigma|_{G_L}$ of $\sigma$  to
$G_L$ is not isomorphic to $\bar \sigma|_{G_L}$. Since
$\rho_{X_1,\ell}|_{G_L}\cong \rho_{X_2,\ell}|_{G_L}$, we know
$$
  \rho_{X_2,\ell}|_{G_L}=\sigma|_{G_L}\oplus\sigma|_{G_L}\oplus \bar
  \sigma|_{G_L}\oplus \bar \sigma|_{G_L}.
$$
Let $(\rho_{X_1,\ell}|_{G_K})^{ss}$ be  the semisimplification of
$\rho_{X_1,\ell}|_{G_K}$.  Since $\Gal(L/K)\cong \Z/4\Z$, by Lemma
\ref{lemma: representation in cyclic}, each $G_K$ irreducible
component of $(\rho_{X_1,\ell}|_{G_K})^{ss}$ is either isomorphic to
$\sigma$ or $\bar \sigma$ up to at most a character of
$G_K$ whose kernel contains $G_L$. Thus we may write
$$
  (\rho_{X_1,\ell}|_{G_K})^{ss}=(\sigma\otimes\phi_1)\oplus
  (\sigma\otimes\phi_2) \oplus (\bar \sigma\otimes\phi_3) \oplus (\bar
  \sigma\otimes\phi_4),
$$
for some  characters $\phi_i$ of $G_K$. Combined with the
above data at $p=3$, we conclude that $\phi_i$ has order 4,
$\phi_3=\phi_1$, and  $\phi_2=\phi_4$. Without loss of generality we
may assume
\begin{equation}\label{equation: alpha beta}
  \phi=\chi_1, \quad \phi^{-1}=\chi_3,
\end{equation}
where $\chi_1,\chi_3$ are defined in \eqref{equation: chi k}. In
summary $$(\rho_{X_1,\ell}|_{G_K})^{ss}=(\sigma\otimes\chi_1)\oplus
(\sigma\otimes\chi_3) \oplus (\bar \sigma\otimes\chi_1) \oplus (\bar
\sigma\otimes\chi_3).$$

Now we turn our attention to general primes $p$ that are congruent to
$3$ modulo $8$. For such a prime $p$, we have $p=a_p^2+2b_p^2$ for some
integers $a_p$ and $b_p$. From Corollary \ref{corollary: L for X2}, we
know that
$$
  \alpha_{p,j}^2=(a_p\pm b_p\sqrt{-2})^2,
$$
each with multiplicity $2$. Then by Lemma \ref{lemma: cases mod 8},
regardless of which $\fp$ lying over $p$,
 we have $\chi_k(\fp)=\pm i$. Thus, $\beta_{p,k_p}^2$ is
equal to one of the numbers
$$
  \pm i(a_p\pm b_p\sqrt{-2})^2,
$$
and consequently, $\beta_{p,k_p}$ is one of
$$
  \zeta_8^m(a_p\pm b_p\sqrt{-2}), \quad m=1,3,5,7.
$$
Because $\beta_{p,k}$, $k=1,\ldots,4$, are Galois conjugates over
$\Q$, we conclude that the $p$-factor of $L(X_2/\Q,s)$ is equal to one
of
$$
  \frac1{1\pm 4b_pp^{-s}+8b_p^2p^{-2s}\pm 4b_pp^{1-3s}+p^{2-4s}}.
$$
In other words, we have
$$
  \# X_2(\F_p)=p+1\pm 4b_p.
$$
On the other hand, because the polynomial $x^6+4x^5+10x^4-20x^2-16x-8$
has an even degree and its leading coefficient is a square, we have
$$
  \# X_2(\F_p)=p+2+\sum_{x=0}^{p-1}\JS{x^6+4x^5+10x^4-20x^2-16x-8}p.
$$
Therefore,
$$
  1+\sum_{x=0}^{p-1}\JS{x^6+4x^5+10x^4-20x^2-16x-8}p=\pm 4b_p.
$$
Together with Lemma \ref{lemma: L for E8}, this yields our main
theorem.
\end{subsection}
\end{section}


\end{document}